\documentclass[12pt,letterpaper]{amsart}
\usepackage{amssymb}
\usepackage{txfonts}
\usepackage{amsmath}
\usepackage{enumerate}

\newtheorem{thm}{Theorem}[section]
\newtheorem{cor}[thm]{Corollary}

\newtheorem{lem}[thm]{Lemma}
\newtheorem{prop}[thm]{Proposition}

\theoremstyle{definition}
\newtheorem{defn}[thm]{Definition}
\theoremstyle{remark}

\numberwithin{equation}{section}

\begin{document}
	\title[]
	{Perelman-type  no breather theorem for noncompact Ricci flows}
	
	\author{Liang Cheng,  Yongjia Zhang}

	\dedicatory{}
	\date{}
	

	
	\thanks{Liang Cheng's Research partially supported by China Scholarship Council, self-determined research funds of CCNU from the colleges' basic research and operation of MOE CCNU19QN075 and Natural Science Foundation of Hubei 2019CFB511
	}
	
	\address{School of Mathematics and Statistics $\&$ Hubei Key Laboratory of Mathematical Sciences, Central  China Normal University, Wuhan, 430079, P.R.China}

	\email{chengliang@mail.ccnu.edu.cn }

\address{School of Mathematics, University of Minnesota, Twin Cities, Minneapolis, MN, 55455, USA}
	
\email{zhan7298@umn.edu}

		\begin{abstract}
In this paper, we first show that a complete shrinking breather with Ricci curvature bounded from below must be a shrinking gradient Ricci soliton. This result has several applications. First, we can classify all complete $3$-dimensional shrinking breathers. Second, we can show that every complete shrinking Ricci soliton with Ricci curvature bounded from below must be gradient---a generalization of Naber's result in \cite{N}. Furthermore, we develop a general condition for the existence of the asymptotic shrinking gradient Ricci soliton, which hopefully will contribute to the study of ancient solutions.
	\end{abstract}

		\maketitle
	
	\section{Introduction}
	
The first application of Perelman's powerful entropy formulas, immediately after they are developed, was the proofs of  no breather theorems \cite{P1}. Perelman proved that every compact shrinking, steady, or expanding breather must be a shrinking, steady, or expanding gradient Ricci soliton, respectively. Let us first of all recall the definitions of breathers and (gradient) Ricci solitons.
	\begin{defn}
A complete solution to the Ricci flow $(M,g(t))_{t\in[0,T]}$ is called a breather, if there exist $0\leq t_1<t_2\leq T$, $\alpha>0$, and a self-diffeomorphism $\phi:M\rightarrow M$, such that
\begin{eqnarray*}
\alpha g(t_1)=\phi^*g(t_2).
\end{eqnarray*}
If $\alpha<0$, $\alpha=1$, or $\alpha>1$, then the breather is called shrinking, steady, or expanding, respectively.
	\end{defn}
	\begin{defn}\label{soliton}
		A Ricci soliton is a tuple $(M, g, X)$, where $(M, g)$ is a smooth Riemannian
		manifold and $X$ is a smooth vector field on $M$ , satisfying
		$$ Ric +\frac{1}{2}\mathcal{L}_Xg =\frac{\lambda}{2}g,$$
where $\lambda$ is a real number. If $\lambda>0$, $\lambda=0$, or $\lambda<0$, then the Ricci soliton is called shrinking, steady, or expanding, respectively. Furthermore, after scaling, the constant $\lambda$ can always be normalized to $1$, $0$, and $-1$ in the shrinking, steady, and expanding cases, respectively (which we shall always do throughout this paper). Finally, if the vector field $X$ is integrable, that is, if there exists a smooth function $f$ such that $X=\nabla f$, then the Ricci soliton is called gradient, and the function $f$ is called the potential function.
	\end{defn}

Ricci solitons generate self-similar solutions to the Ricci flow, called the canonical forms. Let the positive function $\tau(t)$, the one-parameter family of self-diffeomorphisms $\phi_t:M\rightarrow M$, and the evolving metric $g(t)$ be defined as
\begin{eqnarray*}
\tau(t)&=&1+\lambda t,
\\
\frac{\partial}{\partial t}\phi_t(x)&=&\frac{1}{\tau(t)}X(\phi_t(x)),
\\
g(t)&=&\tau(t)\phi_t^*g,
\end{eqnarray*}
then $g(t)$ moves by the Ricci flow equation.

\emph{Remark:} In the construction of the canonical form, the completeness of the vector field $X$ is necessary. By the completeness of a vector field we mean that every one of its flow curves exists for all time. Zhu-Hong Zhang \cite{ZZh} proved the completeness of the vector $\nabla f$ for a gradient Ricci soliton $(M,g,f)$. In the general non-gradient case, the completeness of $X$ is not yet known to the best of our knowledge, though a Ricci soliton without a canonical form is inconceivable. This is an assumption we make in our paper.

If we regard the Ricci flows as orbits in the space
\begin{eqnarray*}
\text{Met}(M)\Big/\text{Diff},
\end{eqnarray*}
where $\text{Met}(M)$ stands for the space of all complete smooth Riemannian metrics on $M$, and $\text{Diff}$ stands for the group of all self-diffeomorphisms and scalings, then the breathers are periodic orbits and the Ricci solitons are static orbits. Therefore, the no breather theorem is tantamount to saying that the periodic orbits must also be static.

Perelman's proofs of no breather theorems, being applications of the $\mathcal{F}$ functional and the $\mathcal{W}$ functional, elegant as they are, rely on the existence of minimizers of these functionals, which is not always true in the general noncompact case, as indicated in \cite{Zhang2} by a counterexample. Qi.S Zhang \cite{Zhang1} also showed that the minimizer of $\mathcal{W}$ exists in the noncompact case assuming certain conditions at infinity, and thereby proved a Perelman-type no shrinking breather theorem for asymptotically flat manifolds with positive scalar curvature.

Lu and Zheng \cite{LZ} came up with the idea of constructing an ancient solution with a given shrinking breather, and presented a new proof of Perelman's no shrinking breather theorem in the compact case. Their method was also capable of dealing with noncompact case under certain additional assumptions. In \cite{Zh}, the second author applied the method of Lu and Zheng in combination with a reduced distance estimate, and extended Perelman's no shrinking breather theorem to the noncompact case assuming only bounded sectional curvature.
	
In this present work, we further extend this Perelman-type no shrinking breather theorem to almost the fullest generality, that is, we assume only a lower bound for the Ricci curvature.

	\begin{thm}\label{main}
		A complete shrinking breather with  Ricci curvature bounded from below must be
		a shrinking gradient Ricci soliton.
	\end{thm}

Let us then discuss some applications of Theorem \ref{main}. First of all, we recall the well-known result that 3-dimensional shrinking gradient Ricci solitons are classified as: $\mathbb{R}^3$, a quotient of $\mathbb{S}^3$, and a quotient of $\mathbb{S}^2\times \mathbb{R}$; see \cite{P1}, \cite{NW} and \cite{CCZ}. On the other hand, since one can always construct an ancient solution with a shrinking breather, it follows from Chen \cite{C} that a $3$-dimensional shrinking breather always has nonnegative sectional curvature;
see Lemma \ref{smooth_ancient_solution} and Corollary \ref{positive_curvature}.
 Taking all these facts into account, we obtain the following classification of 3-dimensional complete shrinking breathers.
	\begin{cor}\label{main2}
		Any $3$-dimensional complete shrinking breather must be
		a gradient shrinking  Ricci soliton, and hence is isometric to $\mathbb{R}^3$, a quotient of $\mathbb{S}^3$, or a quotient of $\mathbb{S}^2\times \mathbb{R}$.
	\end{cor}

Furthermore, by using a result of Munteanu and Wang, we can also derive the following corollary on shrinking breathers with nonnegative sectional curvature; see Corollary 3 in \cite{MW}.

\begin{cor}
A complete shrinking breather with nonnegative sectional curvature can be written as a product $\mathbb{R}^{n-k}\times N^{k}$, where $N$ is a compact shrinking gradient Ricci soliton with nonnegative sectional curvature and positive Ricci curvature.
\end{cor}

Recall that the canonical form of a (not necessarily gradient) Ricci soliton is self-similar, and hence a breather as well. Combining this observation with Theorem \ref{main}, we obtain the following corollary.

\begin{cor}
A complete shrinking Ricci soliton with Ricci curvature bounded from below must be gradient; here we also assume that the vector field $X$ in Definition \ref{soliton} is complete.
\end{cor}
This is a generalization of Naber \cite{N}, where he proved that a shrinking Ricci soliton with bounded sectional curvature must also be gradient.

Besides the applications presented above, we hereby remark that the method we developed might be useful in the study of general ancient solutions. Perelman \cite{P1} first discovered the asymptotic shrinking gradient Ricci soliton for $\kappa$-noncollapsed ancient solutions with bounded and nonnegative curvature operator. Naber \cite{N} proved the same result with curvature operator nonnegativity replaced by the Type I curvature bound. Recall that an ancient solution to the Ricci flow $(M,g(\tau))_{\tau\in[0,\infty)}$, where $\tau$ is the backward time, is called Type I, if its curvature satisfies the following bound
\begin{eqnarray*}
|Rm_{g(\tau)}|\leq\frac{C}{\tau},
\end{eqnarray*}
for all $\tau\in(0,\infty)$, where $C$ is a constant independent of $\tau$. At this point, we do not know many cases in which the asymptotic shrinking gradient Ricci soliton exists. In the smooth category, there are two major cases: the Type I case and the case of ``curvature positivity''. The ``curvature positivity'' condition can be as weak as PIC-2, in which case the existence of asymptotic shrinking gradient Ricci soliton is guaranteed by an extension of Hamilton's trace Harnack proved by Brendle \cite{B}; this, of course, covers Perelman's original bounded and nonnegative curvature operator case. Indeed, in both of the two cases mentioned above, Hamilton's trace Harnack is the crucial element---one may either estimate it using the Type I curvature bound, or directly use its positivity in the case of ``curvature positivity''. As to the non-smooth category, Bamler's recent ground-breaking work shows that every singularity model, being a metric flow, always has an asymptotic metric soliton (which he calls tangent flow at infinity); see \cite{RB1}---\cite{RB3}.

In contradistinction from the former conditions on the global geometry, we consider the condition of the existence of asymptotic shrinking gradient Ricci soliton from another perspective---what is the requirement on the local geometry to effectuate the existence as well as the soliton structure of the blow-down limit? As natural as it appears, we have defined a condition called locally uniformly Type I; see Definition \ref{def_typei}. This condition means that the ancient solution, along a sequence of space-time points with time going to negative infinity, has locally uniform curvature bound,  injectivity radii lower bound, as well as reduced distance bound, and all these bounds are in a ``Type I fashion'', plus a technical assumption---time-wise Ricci curvature lower bound. This condition is sufficient to render certain the existence of an asymptotic shrinking gradient Ricci soliton, and may arguably be a necessary condition also. Indeed, in all the known smooth cases listed above---Type I and ``curvature positivity''---the locally uniformly Type I condition is satisfied.

Other than the existence of the asymptotic shrinking gradient Ricci soliton, we also discover that a locally uniformly Type I ancient solution is $\kappa$-noncollapsed on all scales, where $\kappa$ depends on the Gaussian density of the asymptotic shrinking gradient Ricci soliton. This observation, largely owing to the work of Yokota \cite{YT}, is surprising, since we only assume the noncollapsedness along a sequence of space-time points on a specific sequence of scales, yet we conclude that such an ancient solution is noncollapsed everywhere on all scales. This idea was first used by Perelman \cite{P1} to prove that there exists a universal $\kappa$ for 3-dimensional $\kappa$-solutions.

The key estimates which we obtain are local uniform estimates for the reduced distance as well as its gradient.  These estimates do not follow from a direct generalization of the methods in \cite{P1} and \cite{N}. In particular, we do not have the following inequality
\begin{eqnarray*}
|\nabla l|^2+R\leq\frac{Cl}{\tau},
 \end{eqnarray*}
 which is a consequence of analyzing Hamilton's trace Harnack. The idea of our estimates, is to use local curvature bound to control the reduced distance and its gradient. These estimates, though very coarse, yet are sufficient, in combination with the monotonicity of the reduced volume, to construct the soliton structure on the limit space of the blow-down sequence. Furthermore, unlike the aforementioned cases, we are not able to prove that the reduced volume converges to the Gaussian density of the asymptotic shrinking gradient Ricci soliton; the latter is less than or equal to the limit of the former by Fatou's lemma. This results from the lack of the estimate
\begin{eqnarray*}
l(x,\tau)\sim \frac{1}{\tau} d_{g(\tau)}^2(x,p_{\tau}),
\end{eqnarray*}
where $p_\tau$ is a moving point such that $l(p_\tau,\tau)$ is bounded independent of $\tau$. This estimate is valid in all two cases mentioned above. This no loss of reduced volume property is not necessary for our application, yet we doubt whether it actually could happen.

To conclude the introduction section, we remark that in an upcoming work of the second author with Bamler, Chow, Deng, and Ma \cite{BCDMZ}, it is proved that the canonical form of a noncollapsed steady gradient Ricci soliton with nonnegative Ricci curvature always admits an asymptotic shrinking gradient Ricci soliton.

The present paper is organized as follows. In section 2, we review some well known results about Perelman's $\mathcal{L}$-geometry. In section 3, we construct an ancient solution starting with a given shrinking breather, and subsequently estimate the reduced distance along a sequence of space-time points. In section 4, we define the notion of locally uniformly Type I ancient solution. In section 5, we derive the local uniform estimates for the reduced distance as well as its gradient. In section 6, we prove the existence of asymptotic shrinking gradient Ricci soliton for locally uniformly Type I ancient solutions. In section 7, we discuss the possible loss of reduced volume and the noncollapsedness of locally uniformly Type I ancient solutions.

\emph{Acknowledgement}: The authors would like to thank Professor Japing Wang for bringing this problem to us and many useful discussions. The major work of this paper was accomplished while the first author was visiting University of Minnesota, Twin Cities.

	\section{Preliminaries on Perelman's $\mathcal{L}$-geometry}
	
In this section, we collect some basic properties of Perelman's reduced distance and reduced volume. These properties are well-known, and the reader may easily look them up in the literature.
	Let $(M^n,g(\tau))_{\tau\in[0,T]}$ be a solution to the backward Ricci flow
	\begin{align}\label{Ricci_flow}
	\frac{\partial g}{\partial \tau}=2Ric_{g(\tau)},
	\end{align}
where $\tau$ stands for the backward time. Throughout the whole paper, unless otherwise indicated, we always assume that the (backward) Ricci flows in question admits a lower bound for the Ricci curvature at each time slice.
	
	Perelman's	$\mathcal{L}$-energy for a piecewise $C^1$ curve $\gamma(s):[0,\tau]\to M$, where $\tau\in(0, T]$, is defined as
	\begin{align}\label{l-energy} \mathcal{L}(\gamma(s))=\int^{\tau}_0
	\sqrt{s}\big(R_{g(s)}(\gamma(s))+|\gamma'(s)|_{g(s)}^2\big)ds.
	\end{align}	
	One may view $\gamma$ as a curve in the Ricci flow space-time connecting $(\gamma(0),0)$ and $(\gamma(\tau),\tau)\in M\times[0,T]$, satisfying $(\gamma(s),s)\in M\times\{s\}$ for all $s\in[0,\tau]$.

Let $p\in M$ be a fixed base point. For any $(q,\tau)\in M\times(0,T]$, we define
	$$L(q,\tau)=\inf\limits_{\gamma} \mathcal{L}(\gamma(s)),$$
where the infimum is taken over all  piecewise $C^1$
	curves $\gamma(s):[0,\tau]\to M$ satisfying $\gamma(0)=p$ and
	$\gamma(\tau)=q$. The minimizer is called $\mathcal{L}$-geodesic. By the first variation, the $\mathcal{L}$-geodesic equation is
\begin{eqnarray}\label{l-geodesic}
\nabla_{\gamma'}\gamma'(\tau)-\frac{1}{2}\nabla R(\gamma(\tau),\tau)+\frac{1}{2\tau}\gamma'(\tau)+2Ric(\gamma'(\tau))=0.
\end{eqnarray}
Perelman's
	$l$-function, also known as the reduced distance function, is defined as
	\begin{align}\label{reduced_l} l(q,\tau)=\frac{L(q,\tau)}{2\sqrt{\tau}},
	\end{align}
	and the reduced volume is definded as
	\begin{align}\label{WFRV}
	\mathcal {V}(\tau)=\int_{M} (4\pi \tau)^{-\frac{n}{2}}
	e^{-l(\cdot,\tau)}
	dg(\tau).
	\end{align}
We remark here that  both (\ref{reduced_l}) and (\ref{WFRV}) depend on the choice of $p$, and the point $(p,0)\in M\times[0,T]$ is called the \emph{base point} of the functions $l$ and $\mathcal{V}$. When the base point is understood, we often omit it in the notation. Otherwise the base point is included in the notation as the subindex:
\begin{eqnarray*}
l_{(p,0)} \ \ ,\ \mathcal{V}_{(p,0)}.
\end{eqnarray*}
	
For $v\in T_pM,$ let $\gamma_v$ denote the $\mathcal{L}$-geodesic, that is, a solution to (\ref{l-geodesic}), satisfying $\lim\limits_{s\to 0}\sqrt{s}\gamma'(s) = v$. If $\gamma_v$ exists on $[0,\tau]$, then the $\mathcal{L}$-exponential map $\mathcal{L}\text{exp}^{\tau}_p:T_pM\to M$ at time $\tau$ is defined as
	$$\mathcal{L}\text{exp}^{\tau}_p(v)=\gamma_v(\tau).$$
Let $U(\tau)\subset T_pM\cong\mathbb{R}^n$ denote the maximal domain of $\mathcal{L}\text{exp}^{\tau}_p$.  By applying basic ODE theory to the $\mathcal{L}$-geodesic equation---a linear ODE depending only on the geometry quantities near the $\mathcal{L}$-geodesic---one may obtain that $U(\tau)$ is an open set and that $\mathcal{L}\text{exp}^{\tau}_p$ is a smooth map from $U(\tau)$ to $M$. The injectivity domain at time $\tau$ is defined as
\vspace{0.2in}	

	$\Omega^{T_pM}(\tau)=\Big\{ v\in U(\tau)\ \Big\vert \ \gamma_v|_{[0,\tau]} :[0,\tau] \to M \text{ is the unique minimal } \mathcal{L}-\text{geodesic}
	\text{\ from $p$ to $\gamma_v(\tau)$ }; \text{$\gamma_v(\tau) $ is not conjugate to }p \text{ along\ } \gamma_v.\Big\}.
	$
\vspace{0.2in}

Correspondingly we also define
\vspace{0.2in}
	
	$\Omega(\tau)=\Big\{q\in M\ \Big\vert\  \text{There is a unique minimal } \mathcal{L}-\text{geodesic }\gamma:[0,\tau] \to M
	\text{\ with\ } \gamma(0) = p, \gamma(\tau) = q; q\text{ is not conjugate to } p\text{ along\ } \gamma.\Big\}.
	$
\vspace{0.2in}

It is well known that

$$
\Omega(\tau)=\mathcal{L}\text{exp}^{\tau}_p\big(\Omega^{T_pM}(\tau)\big).
$$
	
The cut-locus
	is defined as
	$$
	C(\tau)=M\backslash \Omega(\tau).
	$$

	In Perelman's study \cite{P1} of the reduced geometry, a general assumption is bounded sectional curvature. However, Ye \cite{Y1} studied the properties of the $l$-function and the reduced volume assuming only a lower bound for the Ricci curvature. We now state these useful results below.
	\begin{thm}[Proposition 2.7, Proposition 2.11, and Lemma 2.14 in \cite{Y1}]\label{Y}
Let $(M^n,g(\tau))_{\tau\in[0,T]}$ be a backward Ricci flow such that the Ricci curvature of each time-slice is bounded from below.  Then the following hold:
\begin{enumerate}
\item For any $(q,\tau)\in M\times(0,T]$, there exists a minimal $\mathcal{L}$-geodesic connecting  $(p,0)$ and $(q,\tau)$, that is,
		$\mathcal{L}\text{exp}^{\tau}_p$ is onto. \item  $L$ is locally Lipschitz in space-time.
\item For each $\tau\in(0,T)$, $C(\tau)\subset M$ is a closed set of zero Remannian measure. Consequently $\cup_{0<\tau< T}C(\tau)\times \{\tau\}$
		is a closed set of zero measure in $M\times(0,T)$.
\end{enumerate}
	\end{thm}

Furthermore, we need the following analytic properties of the $l$-function. They were first discovered by Perelman \cite{P1} and vindicated by Ye \cite{Y1} under the assumption of Ricci curvature lower bound.	
	\begin{thm}[\cite{P1}, see also Lemma 2.19 and Theorem 2.20 in \cite{Y1}]\label{Perelman}
		Let $(M^n,g(\tau))_{\tau\in[0,T]}$ be a backward Ricci flow such that the Ricci curvature of each time-slice is bounded from below. Let $ l$ be the reduced distance function defined in (\ref{reduced_l}). Then on $ \cup_{0<\tau<T}\Omega(\tau)\times \{\tau\}$ it holds that:
		\begin{align}
		&2\frac{\partial l}{\partial \tau}+|\nabla l|^2-R+\frac{l}{\tau}=0,\label{eq_l_1}\\
		&	\frac{\partial }{\partial \tau}l-\Delta l +	|\nabla l|^2-R+\frac{n}{2\tau} \geq 0,\label{eq_l_4}\\
		&	2\Delta l-|\nabla l|^2+R+\frac{l-n}{\tau} \le 0.\label{eq_l_5}
		\end{align}
Furthermore, (\ref{eq_l_4}) and (\ref{eq_l_5}) both hold on $M\times (0,T)$ in the sense of distribution. That is to say, for any $0<\tau_1<\tau_2<T$ and for any nonnegative Lipshcitz function $\phi$ compactly supported on $M\times [\tau_1,\tau_2]$, it holds that
\begin{eqnarray}\label{eq_l_6}
\int_{\tau_1}^{\tau_2}\int_M\Bigg(\nabla l\cdot\nabla\phi+\Big(\frac{\partial}{\partial\tau}l+|\nabla l|^2-R+\frac{n}{2\tau}\Big)\phi\Bigg)dg(\tau)d\tau\geq0,
\end{eqnarray}
and, for any $\tau\in(0,T)$ and any nonnegative Lipshcitz function $\phi$ compactly supported on $M$, it holds that
\begin{eqnarray}\label{eq_l_7}
\int_M\Bigg(-2\nabla l\cdot\nabla\phi+\Big(-|\nabla l|^2+R+\frac{l-n}{\tau}\Big)\phi\Bigg)dg(\tau)\leq 0.
\end{eqnarray}
	\end{thm}
Formulas (\ref{eq_l_6}) and (\ref{eq_l_7}) will be especially important in our future arguments.

The following is a useful consequence of the first variation formula.
	
\begin{lem}[Perelman \cite{P1}] \label	{grad_ll}
Let $\gamma$ be an $\mathcal{L}$-geodesic starting from $(p,0)$. Then, so long as $\tau\in(0,T)$ and $\gamma(\tau)\in \Omega(\tau)$, it holds that
\begin{eqnarray}
\nabla l(\gamma(\tau),\tau)=\gamma'(\tau).
\end{eqnarray}
\end{lem}

The following monotonicity formula is also a well-known result of Perelman.

	\begin{thm}[Perelman\cite{P1}, see also Theorem 4.3 and Theorem 4.5 in \cite{Y1}]\label{Monotonicity}
Let $(M^n,g(\tau)_{\tau\in[0,T]})$ be a backward Ricci flow such that the Ricci curvature of each time-slice is bounded from below. Then
 \begin{enumerate}
 \item $\mathcal {V}(\tau)\leq 1$ for all $\tau\in(0,T]$,
 \item $\mathcal {V}(\tau)$ is non-increasing in $\tau$,
 \end{enumerate}
 where $\mathcal {V}(\tau)$ is the reduced volume defined in (\ref{WFRV}).
	\end{thm}

\section{Construction of the Ancient Solution}
	
In this section, we will use a given shrinking breather to construct an ancient solution. This method was first applied by Lu and Zheng \cite{LZ}, and later adopted by the second author \cite{Zh}. The content of this section is not essentially different from \cite{LZ} and \cite{Zh}, and we are including it for the convenience of the reader. The idea of the construction, so to speak, is to take infinitely many copies of the given shrinking breather, each scaled by a proper factor, and then splice them together head-to-tail. The crucial estimate arising from this construction is this, that if we fix a base point on the given shrinking breather, then the reduced distance evaluated at the base point on each copy is uniformly bounded. In other cases, this fact is also essential to the existence of the asymptotic shrinking gradient Ricci soliton; see \cite{P1} and \cite{N}.  This estimate is the first step to the locally uniform estimates for both the $l$-function and its gradient, which we will present in section 5.

	After rescaling and translating in time, we may assume that the shrinking breather is a backward Ricci flow $(M^n,g_0(\tau))_{\tau\in[0,1]}$, such that there exists $\alpha\in(0,1)$ and a diffeomorphism $\phi:M\to M$ satisfying
	\begin{equation}\label{breather}
	\alpha g_0(1)=\phi^*g_0(0).
	\end{equation}
We first of all observe

\begin{lem}\label{smooth}
$(M,g_0(\tau))_{\tau\in[0,1]}$ is smooth on $[0,1]$. In particular, $g_0(\tau)$ is smooth up to the boundary of the interval $[0,1]$.
\end{lem}

\begin{proof}
According to the localized Shi's estimates \cite{Sh}, we have that $g_0(0)$ must be a smooth Riemannian metric. Since $(M,g_0(0))$ and $(M,g_0(1))$ are isometric, we have that $g_0(1)$ is also smooth. Then by the Appendix in \cite{LT}, we have that $g_0(\tau)$ is smooth up to $\tau=1$.
\end{proof}

	Let
	\begin{equation}\label{tau}
	\tau_i=\sum^i_{k=0}\alpha^{-k},
	\end{equation}
	where we conventionally let $\tau_0=1$. Then there exists $C_0>0$ depending only on $\alpha$ such that
	\begin{equation}\label{ration}
	\alpha^{-i}\le\tau_i\le C_0\alpha^{-i},
	\end{equation}
	for each $i\ge 0$. Next, we define
	\begin{equation*}
	g_i(\tau)=\alpha^{-i}(\phi^i)^*g_0(\alpha^{i}(\tau-\tau_{i-1})), \tau\in [\tau_{i-1},\tau_i].
	\end{equation*}
	Splicing these flows together, we define the evolving metric as
	\begin{equation}\label{defined_ancient_solution}
	g(\tau)=\left\{
	\begin{array}{ll}
	g_0(\tau), \quad & \tau\in [0,\tau_0], \\
	g_i(\tau), \quad &\tau\in [\tau_{i-1},\tau_i].
	\end{array}
	\right.
	\end{equation}
We need to verify that $g(\tau)$ is a smooth Ricci flow. When $g_0(\tau)$ has bounded curvature, this can be easily done by applying the forward and backward uniqueness of the Ricci flow at time $\tau_i$; see \cite{LZ} and \cite{Zh}. In our case, we verify the smoothness of $g(\tau)$ by direct computation.
	
	\begin{lem}\label{smooth_ancient_solution}
		 $g(\tau)$ defined in (\ref{defined_ancient_solution}) is a smooth ancient solution to the Ricci flow on the manifold $M$.
	\end{lem}
	\begin{proof}
		By Lemma \ref{smooth},  we need only to check the smoothness of $g(\tau)$ at each $\tau_i$.
		Applying (\ref{breather}), we get
		\begin{align*}
		&g_i(\tau_{i-1})=\alpha^{-i}(\phi^i)^*g_0(0)
		=\alpha^{-(i-1)}(\phi^{i-1})^*g_0(1)\\
		=&\alpha^{-(i-1)}(\phi^{i-1})^*g_0(\alpha^{i-1}(\tau_{i-1}-\tau_{i-2}))=g_{i-1}(\tau_{i-1}).
		\end{align*}
		Since by (\ref{breather}) and (\ref{defined_ancient_solution}) we have
		\begin{align*}
		\frac{\partial_{+}}{\partial \tau}
		g_i(\tau_{i-1})=(\phi^i)^*\frac{\partial_{+}}{\partial \tau}g_0(0)=(\phi^i)^*Rc(g_0(0))
		\end{align*}
		 and
		\begin{align*}
		\frac{\partial_{-}}{\partial \tau}
		g_{i-1}(\tau_{i-1})=(\phi^{i-1})^*\frac{\partial_{-}}{\partial \tau}g_0(1)=(\phi^{i-1})^*Rc(g_0(1))=(\phi^i)^*Rc(g_0(0)),
		\end{align*}
		where $\displaystyle\frac{\partial_{+}}{\partial\tau}$ and $\displaystyle\frac{\partial_{-}}{\partial\tau}$ stand for the right and the left derivatives, respectively, it then follows that
		$$\frac{\partial_{-}}{\partial \tau}
		g_{i-1}(\tau_{i-1})=\frac{\partial_{+}}{\partial \tau}
		g_{i-1}(\tau_{i-1}).
		$$
		Moreover, because of the smoothness of $g_0(\tau)$ up to $\tau=0$ and $\tau=1$, from the following calculations
		\begin{align*}
		\frac{\partial_{+}^2}{\partial \tau^2}
		g_i(\tau_{i-1})=\alpha^i(\phi^i)^*(\frac{\partial_{+}}{\partial \tau}Rc)(g_0(0)),
		\end{align*}
		\begin{align*}
		\frac{\partial_{-}^2}{\partial \tau^2}
		g_{i-1}(\tau_{i-1})=\alpha^{i-1}(\phi^{i-1})^*(\frac{\partial_{-}}{\partial \tau}Rc)(g_0(1)),
		\end{align*}
		and
		\begin{align*}
		(\frac{\partial_{-}}{\partial \tau}Rc)(g_0(1))=(\Delta_L Rc)(g_0(1))=\alpha \phi^*((\Delta_L Rc)(g_0(0)))=\alpha \phi^*((\frac{\partial_{+}}{\partial \tau}Rc)(g_0(0))),
		\end{align*}
		we conclude that
		\begin{align*}
		\frac{\partial_{+}^2}{\partial \tau^2}
		g_i(\tau_{i-1})=\frac{\partial_{-}^2}{\partial \tau^2}
		g_{i-1}(\tau_{i-1}).
		\end{align*}
		It is straightforward to prove
		$
		\frac{\partial_{+}^k}{\partial \tau^k}
		g_i(\tau_{i-1})=\frac{\partial_{-}^k}{\partial \tau^k}
		g_{i-1}(\tau_{i-1})
		$ for $k\geq 3$
		in the similar way.
	\end{proof}
Chen \cite{C} showed that any ancient solution to the Ricci flow has nonnegative scalar curvature, and any 3-dimensional ancient solution has nonnegative sectional curvature. As a direct application of Lemma \ref{smooth_ancient_solution},
we have the following corollary.
	\begin{cor}\label{positive_curvature}
	Any complete shrinking breather has nonnegative scalar curvature. In particular, any 3-dimensional complete shrinking breather has nonnegative sectional curvature.
\end{cor}	
	
Let us choose a sequence of points on this ancient flow, and estimate the $l$-function at these points. Indeed, these points will serve as the base points of the Cheeger-Gromov-Hamilton convergence \cite{RH5}. Fix an arbitrary point $p_0\in M^n$, and for each $i\geq 0$ define
	\begin{equation} \label{basepoint}
	x_i=\phi^{-{(i+1)}}(p_0).
	\end{equation}
 Let $\sigma:[0,1]\to M^n$ be a smooth curve such that $\sigma(0)=p_0$ and $\sigma(1)=x_0$.
	We define
	\begin{equation}
	\sigma_i(\tau)=\phi^{-{(i+1)}}\circ\sigma\big(\alpha^{i+1}(\tau-\tau_{i})\big), \tau\in
	[\tau_{i},\tau_{i+1}],
	\end{equation}\label{gamma}
	and $\gamma_i:[0,\tau_{i+1}]\to M$ as
	\begin{equation}\label{gamma}
	\gamma_i(\tau)=\left\{
	\begin{array}{ll}
	\sigma(\tau), \quad & \tau\in [0,1], \\
	\sigma_j(\tau), \quad &\tau\in [\tau_{j},\tau_{j+1}].
	\end{array}
	\right.
	\end{equation}
	Since
	\begin{equation*}
	\sigma_i(\tau_{i})=\phi^{-(i+1)}\circ\sigma(0)=\phi^{-i}\circ\sigma(1)=\phi^{-i}\circ\sigma(\alpha^{i}(\tau_{i}-\tau_{i-1}))=\sigma_{i-1}(\tau_{i}),
	\end{equation*}
	we have that $\gamma_i(\tau)$ defined in (\ref{gamma}) is a continuous and piecewise smooth
	curve with $\gamma_i(0)=p_0$ and $\gamma_i(\tau_i)=x_{i+1}$.
	Since for $\tau\in[\tau_{j},\tau_{j+1}]$
	\begin{align*}
	&R(\sigma_j(\tau),\tau)=R_{g_{j+1}(\tau)}(\sigma_j(\tau))\\
	=&
	R_{\alpha^{-(j+1)}(\phi^{j+1})^*g_0(\alpha^{j+1}(\tau-\tau_{j}))}\Big(\phi^{-(j+1)}\circ\sigma\big(\alpha^{j+1}(\tau-\tau_{j})\big)\Big)\\
	=&\alpha^{j+1}
	R_{g_0(\alpha^{j+1}(\tau-\tau_{j}))}(\sigma(\alpha^{j+1}(\tau-\tau_{j})))\le C_1 \alpha^{j+1}\\
	\le& \tau_{j+1}\alpha^{j+1}\frac{C_1}{\tau}\le \frac{B}{\tau},
	\end{align*}
	where $\displaystyle C_1=\max\limits_{\tau\in [0,1]}R_{g_0(\tau)}(\sigma(\tau))$.
	We can then compute the $\mathcal{L}$-energy for $\gamma(\tau)$ defined in (\ref{gamma})
	\begin{align*}
	\mathcal{L}(\gamma_i)&=\mathcal{L}(\sigma)+\sum\limits_{j=1}^{i}
	\int\limits_{\tau_{j}}^{\tau_{j+1}}\sqrt{\tau}(R(\sigma_j(\tau),\tau)+|\sigma_j'(\tau)|^2_{g(\tau)})d\tau\\
	&\le D+\sum\limits_{j=1}^{i}
	\int\limits_{\tau_{j}}^{\tau_{j+1}}\sqrt{\tau}(\frac{B}{\tau}+A\alpha^{j+1})d\tau\\
	&\le D+C\sum\limits_{j=1}^{i}
	\alpha^{-\frac{j+1}{2}},
	\end{align*}
	where $A=\max\limits_{\tau\in [0,1]}|\sigma'(\tau)|^2_{g_0(\tau)}$.
	Hence
	\begin{align*}
	l(x_{i+1},\tau_{i+1})&\le\frac{\mathcal{L}(\gamma_i)}{2\sqrt{\tau_{i+1}}}\\
	&\le \frac{1}{2}D\alpha^\frac{j+1}{2}+\frac{1}{2}C\sum\limits_{j=1}^{i}
	\alpha^{\frac{j+1}{2}}\le C_2,
	\end{align*}
where $C_2$ is a constant independent of $i$.

\section{Locally Uniformly Type I Ancient Solutions}

In this section, we will consider a certain type of ancient solutions to the Ricci flow of which (\ref{defined_ancient_solution}) is a special case. As in \cite{P1} and \cite{N}, we will eventually consider the rescaled sequence of
	backward Ricci flows
\begin{eqnarray}\label{ancient_sequence}
	(M,\tau_i^{-1}g(\tau_i\tau),(x_i,1)) \text{\ \ for\ \ }  \tau\in[0,\infty),
\end{eqnarray}
where $\tau_i$ is as defined in (\ref{tau}) and $g(\tau)$ as defined in (\ref{defined_ancient_solution}). It turns out that, around the space-time base points $(x_i, 1)$, every member of this sequence differs the original shrinking breather only by an isometry and a bounded constant. More specifically,  (\ref{defined_ancient_solution}) implies that  for  $\tau\in [1,\frac{\tau_{i+1}}{\tau_{i}}]$, it holds that $$\tau_i^{-1}g(\tau_i\tau)=\tau_i^{-1}\alpha^{-(i+1)}(\phi^{i+1})^*
	g_0(\alpha^{i+1}(\tau_i\tau-\tau_i)),$$
where the scaling factors $\tau_i^{-1}\alpha^{-(i+1)}$ are uniformly bounded from above and below by constants independent of $i$; indeed, this factor converges to $1$. This means that around the base points $(x_i, 1)$, the scaled metrics $\tau_i^{-1}g(\tau_i\tau)$ have locally uniformly bounded geometry: the curvature norms are bounded in terms of distance from $x_i$ and the injectivity radii at $(x_i, 1)$ are bounded independent of $i$. This is model on which the following definition is based.

\begin{defn}\label{def_typei}
\textbf{(Locally uniformly Type I ancient solution.)} Let $(M,g(\tau))_{\tau\in[0,\infty)}$ be an ancient solution to the Ricci flow, where $\tau$ is the backward time. Fix $p_0\in M$ and let $\{(x_i,\tau_i)\}_{i=1}^\infty\subset M\times(0,\infty)$ be a sequence of space-time points with $\tau_i\nearrow\infty$. $(M,g(\tau))$ is called \emph{locally uniformly Type I} along the space-time sequence $\{(x_i,\tau_i)\}_{i=1}^\infty$, if the following hold.

\begin{enumerate}[(1)]
 \item \begin{eqnarray}\label{def1}
 \limsup_{i\rightarrow\infty}l(x_i,\tau_i)<\infty,
 \end{eqnarray}
where $l$ is the reduced distance function based at $(p_0,0)$.

\item Around the space-time points $(x_i,\tau_i)$, the curvature is locally uniformly Type I. More precisely, there exists a positive function $C:(0,\infty)\rightarrow(0,\infty)$ with the following property: for all $A>0$, there exists $i_0\in\mathbb{N}$, depending on $A$, such that
    \begin{eqnarray}\label{def2}
    \sup_{B_{g(\tau_i)}(x_i,\sqrt{\tau_i}r)\times[\tau_i,2\tau_i]}|Rm_{g(\tau)}|\leq \frac{C(r)}{\tau_i} \text{ for all }i\geq i_0 \text{ and } r\leq A.
    \end{eqnarray}

\item There is time-wise Ricci curvature lower bound for $g(\tau)$. In other words, there exists a positive function $K:[0,\infty)\rightarrow(0,\infty)$, such that
    \begin{eqnarray}\label{def3}
    Ric_{g(\tau)}\geq-K(\tau)g(\tau)
    \end{eqnarray}
    for all $\tau\in[0,\infty)$.

\item $g(\tau)$ is noncollapsed along $(x_i,\tau_i)$. In other words,
    \begin{eqnarray}\label{def4}
    \liminf_{i\rightarrow\infty} \Big((\tau_i)^{-\frac{1}{2}}\text{inj}(g(\tau_i),x_i)\Big)>0,
    \end{eqnarray}
    where $\text{inj}(g, x)$ stands for the injectivity radius of the Riemannian metric $g$ at $x$.

\end{enumerate}

\end{defn}

According to the argument at the beginning of this section, we have that an ancient solution constructed with a shrinking breather with Ricci curvature bounded from below must be locally uniformly Type I.

\begin{prop}\label{somethinguseful}
Let $(M,g_0(\tau))_{\tau\in[0,1]}$ be a shrinking breather as defined in (\ref{breather}) with Ricci curvature bounded from below, and $(M,g(\tau))_{\tau\in[0,\infty)}$ the ancient solution as defined in (\ref{defined_ancient_solution}). Let $p_0\in M$ be a fixed point, $x_i=\phi^{-{(i+1)}}(p_0)$, and $\displaystyle\tau_i=\sum^i_{k=0}\alpha^{-k}$. Then, $(M,g(\tau))_{\tau\in[0,\infty)}$ is a locally uniformly Type I ancient solution along $\{(x_i,\tau_i)\}_{i=1}^\infty$.

Furthermore, the scaled flows $\displaystyle \big(M,\tau_i^{-1}g(\tau\tau_i),(x_i,1)\big)_{\tau\in[1,\frac{\tau_{i+1}}{\tau_i}]}$ differ the original shrinking breather $\displaystyle \big(M,g_0(\tau-1),(p_0,1)\big)_{\tau\in[1,2]}$ only by a scaling factor $\tau_i^{-1}\alpha^{-(i+1)}$ and a base-point-preserving diffeomorphism $\phi^{-(i+1)}$.
\end{prop}

\emph{Remark}: Though the ratios $\displaystyle\frac{\tau_{i+1}}{\tau_{i}}$ are not necessarily equal to $2$, and item (2) of Definition \ref{def_typei} requires locally uniformly bounded curvature on interval $[\tau_i,2\tau_i]$, yet one may verify from  (\ref{defined_ancient_solution}) that in the special case of an ancient solution constructed from a shrinking breather, (\ref{def2}) is indeed satisfied. Or one may simply take $\alpha$ to be, say, $\frac{1}{4}$, for simplicity.

	\section{Estimates on the $l$-distance}
	
In this section and the next, we proceed to prove that a locally uniformly Type I ancient solution always has an asymptotic shrinking gradient Ricci soliton. According to Perelman's \cite{P1} idea of asymptotic shrinking gradient Ricci soliton, we aim at showing that the scaled sequence of backward Ricci flows,
\begin{eqnarray*}
\Big\{\big(M,g_i(\tau),x_i,l_i\big)_{\tau\in[1,2]}\Big\}_{i=1}^\infty,
\end{eqnarray*}
where $g_i(\tau)=\tau_i^{-1}g(\tau\tau_i)$, $l_i(\tau)=l(\tau\tau_i)$, and all other notations are as in the definition in previous section, after passing to a subsequence, converges to a shrinking gradient Ricci soliton. To this end, we need not only the locally uniform geometric bound for the scaled sequence (\ref{ancient_sequence}), but also the locally uniform bound for the $l$-function and its gradient, as in \cite{P1} and \cite{N}; this is the goal of this section.

In our case, we have neither nonnegative curvature operator nor type-I bound. These two assumptions are made by Perelman and Naber to deal with Hamilton's trace Harnack, respectively. Nevertheless, the bound of $l$-function at $(x_i,\tau_i)$ and the locally uniform Type I curvature bound turn out to be sufficient. For the locally uniform estimate on $l$, the idea of triangle inequality is very useful: one may use as a test curve the $\mathcal{L}$-geodesic from the base point to $(x_i,\tau_i)$ concatenated with another curve from $(x_i,\tau_i)$ to the target point, where the bound of the former is a consequence of section 3, and the latter can be estimated using the local bound on curvature. For the gradient estimate of $l$, we use the following idea: Let $(x,\tau)$ be a point not too far away from $(x_i,\tau_i)$, and let $\gamma$ be the minimal $\mathcal{L}$-geodesic from $(p_0,0)$ to $(x,\tau)$. If $\gamma$ runs too fast as it approaches $(x,\tau)$, then the latter must be quite far away from $(x_i,\tau_i)$, contradicting our assumption on $(x,\tau)$. In combination with (\ref{grad_ll}) we obtain the gradient estimate.

Let $\displaystyle (M,g(\tau))_{\tau\in[0,\infty)}$, $p_0\in M$, and $\displaystyle \{(x_i,\tau_i)\}_{i=1}^\infty \subset M\times(0,\infty)$ be as in the Definition \ref{def_typei}. Let us consider the following scaled backward Ricci flows
\begin{eqnarray}\label{bound1}
\Big(M,g_i(\tau),(x_i,1)\Big)_{\tau\in[0,\infty)},
\end{eqnarray}
where $g_i(\tau)=\tau_i^{-1}g(\tau\tau_i)$. Then, let
\begin{eqnarray}\label{l}
l_i(\tau):=l(\tau\tau_i) \text{ for } \tau\in(0,\infty)
 \end{eqnarray}
where $l$ is the reduced distance from the base point $(p_0,0)$ with respect to the ancient flow $g(\tau)$. Indeed, $l_i$ is the reduced distance from $(p_0,0)$ with respect to the flow $g_i(\tau)$.

Next, we interpret (\ref{def1})---(\ref{def4}) in terms of the scaled flows (\ref{bound1}). First of all, (\ref{def1}) becomes
 \begin{eqnarray}\label{l-bound}
 l_i(x_i,1)=l(x_i,\tau_i)\leq C\text{ for all }i\geq 0,
 \end{eqnarray}
  where $C$ is a constant independent of $i$.

The locally uniform curvature bound (\ref{def2}) can be interpreted as follows. There exists a positive function $C_0:(0,\infty)\rightarrow(0,\infty)$, such that the following holds: for all $r>0$,
\begin{eqnarray}\label{bound2}
|Rm_{g_i(\tau)}|\leq C_0(r) \text{ on } B_{g_i(1)}(x_i,r)\times[1,2] \text{ for all } i \text{ large enough.}
\end{eqnarray}

The time-wise Ricci curvature lower bound (\ref{def3}) now becomes the following. There exists a positive sequence $\{K_i\}_{i=1}^\infty$, such that
\begin{eqnarray}\label{bound3}
Ric_{g_i(\tau)}\geq -K_ig_i(\tau) \text{ for all } i\text{ and for all } \tau\in[1,2]\text.
\end{eqnarray}

The injectivity radii condition (\ref{def4}) becomes
\begin{eqnarray}\label{bound_a}
\liminf_{i\rightarrow\infty} \Big(\text{inj}(g_i(1),x_i)\Big)>0.
\end{eqnarray}

Finally, by using (\ref{bound2}), we may apply a straightforward distance distortion estimate to obtain the following. For all $r>0$,
\begin{eqnarray}\label{bound4}
B_{g_i(\tau)}(x_i,r)&\subset& B_{g_i(1)}\big(x_i,C_0(r)\big),\\
B_{g_i(1)}(x_i,r)&\subset& B_{g_i(\tau)}\big(x_i,C_0(r)\big),\label{bound5}
\end{eqnarray}
for all $\tau\in[1,2]$ and for all $i$ large enough, where $C_0(r)$ is a positive function of $r$. The function $C_0$ in (\ref{bound4}), (\ref{bound5}), and (\ref{bound2}) are not necessarily the same. Nevertheless, we may always use the same notation for the sake of convenience.

The above setting is what we work with in the current section. Indeed, (\ref{bound3}) and (\ref{bound_a}) are not needed until the next section. The following is the main result of this section

	\begin{prop}\label{l-estiamtes}
		Consider the sequence of backward Ricci flows in (\ref{bound1}) with locally uniformly curvature bound (\ref{bound2}). Let $l_i$ be the functions defined in (\ref{l}), satisfying (\ref{l-bound}). Then the following hold. For any $\displaystyle\varepsilon\in\left(0,\frac{1}{4}\right)$ there exists a positive function $C(\epsilon,\cdot):(0,\infty)\rightarrow(0,\infty)$, such that for any $r>0$ it holds that
		\begin{enumerate}
		\item $\displaystyle 0\leq l_i(x,\tau)\leq C(\epsilon,r)$ for $(x,\tau)\in B_{g_i(1)}(x_i,r)\times [1+\varepsilon,2]$ for all $i$ large enough
		\vspace{0.05in}
		\item $\displaystyle\left|\frac{\partial l_i}{\partial \tau}(x,\tau)\right|+|\nabla l_i(x,\tau)| \leq C(\epsilon,r)$ for  $(x,\tau)\in \big(B_{g_i(1)}(x_i,r) \times [1+\varepsilon,2-\varepsilon]\big)\cap\big( \cup_{\tau\in[1+\varepsilon,2-\varepsilon]}\Omega(\tau)\times\{\tau\}\big)$ and for all $i$ large enough.
\end{enumerate}
Here $\Omega(\tau)$ is as defined in section 2.
	\end{prop}

	\begin{proof}
In the following proof of both (1) and (2), we assume that $\displaystyle\varepsilon\in\left(0,\frac{1}{4}\right)$ is a fixed number.	 Furthermore, whenever we fix an $r>0$, we always assume $i$ is taken large enough so that the formulas (\ref{bound2}), (\ref{bound4}), and (\ref{bound5}) can be applied.

(1) The positivity of $l$-function follows from the positivity of the scalar curvature on ancient solutions; see \cite{C}.

Fix a $\tau\in[1+\varepsilon,2]$ and an $r>0$. By (\ref{l-bound}) we have that there exists a curve $\alpha_i:[0,1]\to M$ from $(p_0,0)$ to $(x_i,1)$
		such that
		\begin{equation}
		\mathcal{L}_i(\alpha_i)\le C.
		\end{equation}
		Let $\zeta_i(s):[0,\tau-1]$ be the minimal geodesic from $x_i$ to $x$ with respect to $g_i(\tau)$, where $\displaystyle x\in B_{g_i(\tau)}(x_i,r)\subset B_{g_i(1)}\big(x_i,C_0(r)\big)$; the latter inclusion is by (\ref{bound4}).
Set
		\begin{equation}
		\eta_i(s)=\left\{
		\begin{array}{ll}
		\alpha_i(s), \quad &s\in [0,1], \\
		\zeta_i(s-1), \quad &s\in (1,\tau].
		\end{array}
		\right.
		\end{equation}
		By (\ref{bound2}), we have $\displaystyle|\eta'_i(s)|^2_{g_i(s)}\leq e^{2nC_1(r)\tau} |\zeta'_i(s-1)|^2_{g_i(\tau)}$ for $s\in (1,\tau]$. This is because
\begin{eqnarray*}
&\zeta_i\subset B_{g_i(\tau)}(x_i,r)\subset B_{g_i(1)}\big(x_i,C_0(r)\big),&
\\
&\max_{{\zeta_i}\times[1,2]}|Rm_{g_i}|\leq C_1(r),&
\end{eqnarray*}
 where $C_1(r)$ is a constant depending on $r$. Hence
		\begin{align*}
		&L_i(x,\tau)\\
		\leq & \mathcal{L}_i(\eta_i(s))\\
		=& \mathcal{L}_i(\alpha_i)
		+\int^{\tau}_{1} \sqrt{s}\Big(R_{g_i(s)}(\eta_i(s))+|\eta_i'(s)|^2_{g_i(s)}\Big)ds\\
		\le& \mathcal{L}_i(\alpha_i)
		+\int^{\tau}_{1} \sqrt{s}\Big(R_{g_i(s)}(\zeta_i(s-1))+e^{2nC_1(r)\tau}|\zeta'_i(s-1)|^2_{g_i(\tau)}\Big)ds\\
		\leq &\mathcal{L}_i(\alpha_i)+2 C_1(r)+\frac{2}{3}e^{2nC_1(r)\tau}\frac{r^2}{(\tau-1)^2}(\tau^{\frac{3}{2}}-1)\\
        \leq&2C+2C_1(r)+4e^{2nC_1(r)\tau}\varepsilon^{-2}.
		\end{align*}
Taking (\ref{bound5}) into account, this finishes the proof of part (1).

		(2)  Let $\gamma_i$ be the minimal $\mathcal{L}-$geodesic from $(p_0,0)$ to $(x,\tau)$ with respect to the backward Ricci flow $g_i$, where  $x\in (B_{g_i(\tau)}(x_i,r)\cap \Omega(\tau)) $ and $\tau\in[1+\varepsilon,2-\varepsilon]$.
		Since
$$B_{g_i(\tau)}(x,r)\subset B_{g_i(\tau)}(x_i,2r)\subset B_{g_i(1)}\big(x_i,C_0(2r)\big),$$
 we have by (\ref{bound2}) that
		\begin{eqnarray}\label{curv}
\max_{ B_{g_i(\tau)}(x,r)\times [1,2]}|Rm_{g_i}|\leq C_2(r).
\end{eqnarray}
 Let $q_i=\gamma_i(\tau_{q_i})$ be defined as
\begin{eqnarray*}
\tau_{q_i}=\inf\Big\{s\in[0,\tau]\ \Big\vert\ \gamma_i|_{[s,\tau]}\subset B_{g_i(\tau)}(x,r)\Big\}.
\end{eqnarray*}

\noindent \underline{Claim}: There exists a positive number $c_1$ depending on $r$ and $\varepsilon$, but independent of $i$, such that
		$\sqrt{\tau}-\sqrt{\tau_{q_i}}\geq c_1$ for all $i$ large enough.
\begin{proof}[Proof of the claim]
We argue by contradiction. Let $x_i'\in B_{g_i(\tau)}(x_i,r)$ be a sequence of counterexamples. In other words, let $\gamma_i$ be minimal $\mathcal{L}$-geodesic from $(p_0,0)$ to $(x_i',\tau)$ with respect to the flow $g_i$ and
 \begin{eqnarray*}
\tau_{q_i}=\inf\Big\{s\in[0,\tau]\ \Big\vert\ \gamma_i|_{[s,\tau]}\subset B_{g_i(\tau)}(x_i',r)\Big\}.
\end{eqnarray*}
But $\sqrt{\tau}-\sqrt{\tau_{q_i}}\to 0$ as $i\to \infty$.  We can then assume that $\tau_{q_i}\geq 1$ when $i$ is large. As before, since
$$\gamma_i|_{[\tau_{q_i},\tau]}\subset B_{g_i(\tau)}(x_i',r)\subset B_{g_i(\tau)}(x_i,2r)\subset B_{g_i(1)}\big(x_i,C_0(2r)\big),$$
we have
\begin{eqnarray}\label{curv_1}
\max_{ \gamma_i |_{[\tau_{q_i},\tau]}\times[1,2]}|Rm_{g_i}|\leq\max_{ B_{g_i(\tau)}(x_i',r)\times [1,2]}|Rm_{g_i}|\leq C_2(r).
\end{eqnarray}
		
Let us change the variable so that $\beta_i(\sigma)=\gamma_i(\sigma^2)$. We can then write the $\mathcal{L}$-energy of $\gamma_i$ as
		$$\mathcal{L}_i(\gamma_i)=\int^{\sqrt{\tau}}_0 \left(2\sigma^2 R_{g_i(\sigma^2)}(\beta_i(\sigma))+\frac{1}{2}|\beta_i'(\sigma)|^2_{g_i(\sigma^2)}\right)d\sigma.$$
Then we have
		\begin{align*}
		\int^{\sqrt{\tau}}_{\sqrt{\tau_{q_i}}} \frac{1}{2}|\beta_i'(\sigma)|^2_{g_i(\sigma^2)}d\sigma&=\mathcal{L}_i(\gamma_i|_{[\tau_{q_i},\tau]})-
		\int^{\sqrt{\tau}}_{\sqrt{\tau_{q_i}}} 2\sigma^2 R_{g_i(\sigma^2)}(\beta_i(\sigma))d\sigma\\
		&\le\mathcal{L}_i(\gamma_i)=2\sqrt{\tau}l_i(x_i',\tau)\\
		&\le
		C_3(r),
		\end{align*}
where we have also used part (1) as well as the positivity of scalar curvature on ancient Ricci flows. It follows from the definition of $q_i$ and (\ref{curv_1}) that
		\begin{align*}
		r^2=d^2_{g_i(\tau)}(x_i',q_i)&\le\left(\int^{\sqrt{\tau}}_{\sqrt{\tau_{q_i}}} |\beta_i'(\sigma)|_{g_i(\tau)}d\sigma\right)^2\\
&\le e^{2n C_2(r)\tau}\left(\int^{\sqrt{\tau}}_{\sqrt{\tau_{q_i}}} |\beta_i'(\sigma)|_{g_i(\sigma^2)}d\sigma\right)^2\\
		&\le (\sqrt{\tau}-\sqrt{\tau_{q_i}})e^{2n C_2(r)\tau}\int^{\sqrt{\tau}}_{\sqrt{\tau_{q_i}}} |\beta_i'(\sigma)|^2_{g_i(\sigma^2)}d\sigma\\
		&\le C_4(r)(\sqrt{\tau}-\sqrt{\tau_{q_i}});
		\end{align*}
a contradiction against $\sqrt{\tau}-\sqrt{\tau_{q_i}}\to 0$. This proves the claim.
\end{proof}
		
		Let us continue working with the minimal $\mathcal{L}$-geodesic connecting $(p_0,0)$ and $(x,\tau)$, where $x\in B_{g_i(\tau)}(x_i,r)$. After the change of the variable $\beta_i(\sigma)=\gamma_i(\sigma^2)$, the $\mathcal{L}-$geodesic equation (\ref{l-geodesic}) becomes
		\begin{equation*}
		\nabla_{\beta_i'(\sigma)} \beta_i'(\sigma)-2\sigma^2\nabla R+4\sigma Rc(\beta_i'(\sigma))=0,
		\end{equation*}
		So we have
		\begin{eqnarray*}
		\frac{d}{d\sigma}|\beta_i'(\sigma)|_{g_i(\sigma^2)}^2&=&4\sigma^2 \nabla R\cdot \beta_i'-4\sigma Ric_{g_i}(\beta_i',\beta_i')
\\
&\leq&C_5(r)\left(|\beta_i'(\sigma)|_{g_i(\sigma^2)}+|\beta_i'(\sigma)|_{g_i(\sigma^2)}^2\right)
\\
&\leq&C_5(r)\left(1+2|\beta_i'(\sigma)|_{g_i(\sigma^2)}^2\right),
		\end{eqnarray*}
for all $\displaystyle\sigma\in\left[\sqrt{\tau_{q_i}'},\sqrt{\tau}\right]$, where $\tau_{q_i}'=\max\{\tau_{q_i},1\}$; we have made use of the fact $\displaystyle\beta_i|_{[\sqrt{\tau_{q_i}'},\sqrt{\tau}]}\subset B_{g_i(\tau)}(x,r)$, (\ref{curv}), and Shi's local gradient estimates \cite{Sh}. Indeed, (\ref{curv}) is still true if we replace $B_{g_i(\tau)}(x,r)$ by $B_{g_i(\tau)}(x,2r)$, with possibly a different $C_2$. Therefore Shi's estimates can be implemented. Integrating the above inequality, we obtain
		\begin{eqnarray}\label{beta}
|\beta_i'(\sigma)|^2_{g_i(\sigma^2)}\ge c_6(r)|\beta_i'(\sqrt{\tau})|^2_{g_i(\tau)}-C_6(r)
		\end{eqnarray}
		for $\sigma\in[\sqrt{\tau_{q_i}'},\sqrt{\tau}]$ .
		Then, by part (1), the claim, and (\ref{beta}), we have
		\begin{align*}
		C(r)\ge\mathcal{L}_i(\gamma_i|_{[\tau_q',\tau]})&=\int^{\sqrt{\tau}}_{\sqrt{\tau_q'}} \frac{1}{2}|\beta_i'(\sigma)|^2_{g_i(\sigma^2)}d\sigma+
		\int^{\sqrt{\tau}}_{\sqrt{\tau_q'}} 2\sigma^2 R_{g_i}(\beta_i(\sigma),\sigma^2)d\sigma\\
		&\ge c_7(r)|\beta_i'(\sqrt{\tau})|_{g_i(\tau)}^2-C_7(r)
		\end{align*}
		Using Lemma \ref{grad_ll}, we obtain
		\begin{align*}
		|\nabla l_i(x,\tau)|=|\gamma_i'(\tau)|=\frac{|\beta_i'(\sqrt{\tau})|}{2\sqrt{\tau}}\le C_8(r).
		\end{align*}
Finally, by (\ref{eq_l_1}), (\ref{bound2}), and part (1), we obtain
		\begin{align*}
		\left|\frac{\partial l_i}{\partial \tau}(x,\tau)\right|\le C_9(r);
		\end{align*}
this completes the proof of the theorem.
	\end{proof}
	
	\section{Asymptotic Shrinking Gradient Ricci Soliton}

After the preparation in the previous section, we are ready to show that a locally uniformly Type I ancient solution has an asymptotic shrinking gradient Ricci soliton.

\begin{thm}\label{asymptotic shrinker}
Let $(M,g(\tau))_{\tau\in[0,\infty)}$ be an ancient solution to the Ricci flow, where $\tau$ is the backward time. Let $p_0\in M$ and $\{(x_i,\tau_i)\}_{i=1}^\infty$ be a space-time sequence such that $\tau_i\nearrow\infty$. Assume that $(M,g(\tau))_{\tau\in[0,\infty)}$ is locally uniformly Type I along $\{(x_i,\tau_i)\}_{i=1}^\infty$. Then the sequence of tuples
\begin{eqnarray}\label{sequenceoftuples}
\Big\{\big(M,g_i(\tau),x_i,l_i\big)_{\tau\in[1,2]}\Big\}_{i=1}^\infty,
\end{eqnarray}
where $g_i(\tau)=\tau_i^{-1}g(\tau\tau_i)$, $l_i(\tau)=l(\tau\tau_i)$, and $l$ is the reduced distance function based at $(p_0,0)$, after passing to a subsequence, converges to the canonical form of a shrinking gradient Ricci soliton
\begin{eqnarray*}
\big(M_\infty,g_\infty(\tau),x_\infty,l_\infty\big)_{\tau\in(1,2)}
\end{eqnarray*}
with $l_\infty$ being the potential function, that is,
\begin{eqnarray*}
Ric_{g_\infty(\tau)}+\nabla^2l_\infty-\frac{1}{2\tau}g_{\infty}(\tau)=0,
\end{eqnarray*}
for all $\tau\in(1,2)$. Here the Ricci flows $\big(M,g_i(\tau),x_i\big)_{\tau\in[1,2]}$ converge in the smooth Cheeger-Gromov-Hamilton sense \cite{RH5}, and the functions $l_i$ converge in the $C_{\text{loc}}^{0,\alpha}$ sense as well as in the weak $*W_{\text{loc}}^{1,2}$ sense.
\end{thm}

Let us consider the sequence of tuples (\ref{sequenceoftuples}). Because the locally uniform Type I assumption, (\ref{bound2})---(\ref{bound5}) are still valid for $g_i$ and $l_i$. Furthermore, Proposition \ref{l-estiamtes} also implies a locally uniform bound for $l_i$ as well as its gradient around $(x_i,1)$. This is the setting which we work with subsequently in this section.

The following convergence of underlying Ricci flows is an immediate consequence of (\ref{bound2}) and (\ref{bound_a}).

\begin{lem}
After passing to a subsequence, the sequence of backward Ricci flows
\begin{eqnarray*}
\Big\{\big(M,g_i(\tau),(x_i,1)\big)_{\tau\in[1,2]}\Big\}_{i=1}^\infty
\end{eqnarray*}
converges in the smooth Cheeger-Gromov-Hamilton sense \cite{RH5} to a smooth backward Ricci flow
\begin{eqnarray*}
\big(M_\infty,g_\infty(\tau),(x_\infty,1)\big)_{\tau\in[1,2)}.
\end{eqnarray*}
\end{lem}

\emph{Remark}: It is not known whether $g_\infty$ is defined at $\tau=2$, because Shi's estimates \cite{Sh} cannot be applied at this instance. Nevertheless, in the case of shrinking breather, that is, if the ancient solution is defined as in (\ref{defined_ancient_solution}), $g_\infty$ exists and is smooth up to $\tau=2$.

In the definition of Cheeger-Gromov-Hamilton convergence, there are diffeomorphism involved:
\begin{eqnarray}\label{diffeo}
\Phi_i: M_\infty\times[1,2)\supset U_i\rightarrow V_i\subset M\times[1,2],
\end{eqnarray}
where $\{U_i\}_{i=1}^\infty$  is a sequence of open sets exhausting the space-time $M_\infty\times[1,2)$, with $(x_\infty,1)\in U_i$ and $\Phi_i(x_\infty,1)=(x_i,1)$ for all $i$ . The composite functions
\begin{eqnarray*}
l_i\circ\Phi_i
\end{eqnarray*}
are therefore functions on $M_\infty\times[1,2)$. In the following argument, when we are referring to $l_i$ as a function on $M_\infty\times[1,2)$, this composition is implicitly understood, though we still write it as $l_i$ for simplicity.

Let us then recall some basic properties of $l_i$. We define
\begin{eqnarray*}
\mathcal{V}_i(\tau):=\int_M(4\pi \tau)^{-\frac{n}{2}}e^{-l_i}dg_i(\tau), \ \tau\in[1,2].
\end{eqnarray*}
By the monotonicity property of reduced volume (Theorem \ref{Monotonicity}(2)), we have
\begin{eqnarray}\label{final mono}
\lim_{\i\rightarrow\infty}\mathcal{V}_i(\tau)=V_\infty \text{ uniformly for }\tau\in[1,2],
\end{eqnarray}
where $V_\infty\in[0,1]$ is a constant. We will see in section 7 that $V_\infty>0$, and $V_\infty=1$ implies that the Ricci flow is a static Euclidean space; see \cite{YT}. For the moment, we only need the fact that $V_\infty$ is a constant. Furthermore, $l_i$ also satisfies the formulas (\ref{eq_l_1})---(\ref{eq_l_7}).

\begin{lem}
After passing to a subsequence, $\{l_i\}_{i=1}^\infty$ converges in the $C_{\text{loc}}^{0,\alpha}\big(M_\infty\times(1,2)\big)$ sense as well as in the weak $*W_{\text{loc}}^{1,2}\big(M_\infty\times(1,2)\big)$ sense to a locally Lipschitz function $l_\infty$; here $\alpha\in(0,1)$ is an arbitrarily fixed number. By passing to a further subsequence, we may also assume that $\displaystyle(4\pi\tau)^{-\frac{n}{2}}e^{-l_i}\rightarrow(4\pi\tau)^{-\frac{n}{2}}e^{-l_\infty}$ in the same sense. In particular, $(4\pi\tau)^{-\frac{n}{2}}e^{-l_\infty}$ is a positive function.
\end{lem}
\begin{proof}
This is a direct consequence of Proposition \ref{l-estiamtes}.
\end{proof}

Subsequently, our main goal for this section is to prove the following.

\begin{prop} \label{propheat}
$(4\pi\tau)^{-\frac{n}{2}}e^{-l_\infty}$  is a smooth solution to the conjugate heat equation $\displaystyle \frac{\partial}{\partial\tau}u-\Delta u+Ru=0$ on $M_\infty\times(1,2)$.
\end{prop}

We break down the proof of Proposition \ref{propheat} into several steps. First of all, the existence of a time-independent cut-off function with small gradient is merely a consequence of (\ref{bound3}).

\begin{lem}\label{cutoff}
For  any positive number $A$, there exists a smooth function $\varphi_{A,i}:M\rightarrow [0,1]$ satisfying the following conditions:
\begin{enumerate}[(1)]
\item $\varphi_{A,i}$ is supported in $B_{g_i(2)}(x_i,2A)$,
\item $\varphi_{A,i}\equiv 1$ on $B_{g_i(2)}(x_i,A)$,
\item $ |\nabla\varphi_{A,i}|_{g_i(\tau)}\leq\frac{C_{K_i}}{A}\  \text{ for all }\ \tau\in[1,2]$,
\end{enumerate}
where the constant $C_{K_i}$ depending only on the Ricci curvature lower bound $K_i$.
In particular, $C_{K_i}$ is independent of $A$.
\end{lem}
\begin{proof}
Let $\chi(t):[0,+\infty)\to [0,1]$ be a smooth function satisfying
$\chi(t)=1$ for $t\in[0,1]$, $\chi(t)=0$ for $t \ge 2$ and $|\chi'(t)|\le 2$.
Then one may easily verify that  $\displaystyle\varphi_{A,i}(x)=\chi\left(\frac{d_{g_i(2)}(x_i,x)}{A}\right)$ satisfies the lemma.
\end{proof}
This cut-off function will be used repeatedly in the following arguments.

\begin{lem}\label{grad}
There exists a constant $C_0$ independent of $i$, such that
\begin{eqnarray}
\int_M|\nabla l_i|^2_{g_i(\tau)}(4\pi\tau)^{-\frac{n}{2}}e^{-l_i}dg_i(\tau)\leq C_0,
\end{eqnarray}
for all $i$ and for all $\tau\in [1,2]$.
\end{lem}
\begin{proof}
Rewriting (\ref{eq_l_7}) for $l_i$, we have
\begin{eqnarray*}
\int_M\Big(-2\nabla l_i\cdot\nabla\phi-|\nabla l_i|_{g_i(\tau)}^2\phi\Big)dg_i(\tau)&\leq& -\int_M\Big(R_{g_i}+\frac{l_i-n}{\tau}\Big)\phi dg_i(\tau)
\\
&\leq&\frac{n}{\tau}\int_M \phi dg_i(\tau)
\end{eqnarray*}
for all $\tau\in[1,2]$, where in the second inequality we have used the facts that $R_{g_i}\geq 0$ and that $l_i\geq 0$; see Proposition \ref{l-estiamtes}(1).

Let us take
\begin{eqnarray*}
\phi:=\varphi_{A,i}^2(4\pi\tau)^{-\frac{n}{2}}e^{-l_i},
\end{eqnarray*}
where $\varphi_{A,i}$ is as defined in Lemma \ref{cutoff}. This is valid, since $l_i$ is locally Lipschitz, and $\varphi_{A,i}$ is smooth as well as compactly supported. Then the above inequality becomes
\begin{eqnarray*}
\int_M\Big(\varphi_{A,i}^2|\nabla l_i|^2_{g_i(\tau)}-4\langle\nabla l_i,\nabla\varphi_{A,i}\rangle\varphi_{A,i}\Big)(4\pi\tau)^{-\frac{n}{2}}e^{-l_i}dg_i(\tau)\leq\frac{n}{\tau}\int_M\varphi_{A,i}^2(4\pi\tau)^{-\frac{n}{2}}e^{-l_i}dg_i(\tau).
\end{eqnarray*}
Hence we have
\begin{eqnarray*}
\int_M\varphi_{A,i}^2|\nabla l_i|^2_{g_i(\tau)}(4\pi\tau)^{-\frac{n}{2}}e^{-l_i}dg_i(\tau)&\leq&\frac{n}{\tau}\int_M\varphi_{A,i}^2(4\pi\tau)^{\frac{n}{2}}e^{-l_i}dg_i(\tau)
\\
&&+4\int_M\Big(\langle\nabla l_i,\nabla\varphi_{A,i}\rangle\varphi_{A,i}\Big)(4\pi\tau)^{\frac{n}{2}}e^{-l_i}dg_i(\tau)
\\
&\leq&\frac{n}{\tau}\int_M\varphi_{A,i}^2(4\pi\tau)^{\frac{n}{2}}e^{-l_i}dg_i(\tau)
\\
&&+\frac{1}{2}\int_M\varphi_{A,i}^2|\nabla l_i|^2_{g_i(\tau)}(4\pi\tau)^{\frac{n}{2}}e^{-l_i}dg_i(\tau)
\\
&&+8\int_M|\nabla\varphi_{A_i}|^2_{g_i(\tau)}(4\pi\tau)^{\frac{n}{2}}e^{-l_i}dg_i(\tau),
\end{eqnarray*}
and subsequently
\begin{eqnarray*}
\int_M\varphi_{A,i}^2|\nabla l_i|^2_{g_i(\tau)}(4\pi\tau)^{-\frac{n}{2}}e^{-l_i}dg_i(\tau)&\leq& \left(\frac{16C_{K_i}^2}{A^2}+\frac{2n}{\tau}\right)\int_M(4\pi\tau)^{-\frac{n}{2}}e^{-l_i}dg_i(\tau)
\\
&=&\left(\frac{16 C_{K_i}^2}{A^2}+\frac{2n}{\tau}\right)\mathcal{V}_i(\tau).
\end{eqnarray*}
Taking $A\rightarrow\infty$ completes the proof; note that $\mathcal{V}_i(\tau)\leq 1$ by Theorem \ref{Monotonicity}.
\end{proof}

Next, we consider the other distributional inequality (\ref{eq_l_6}) satisfied by $l_i$. Replacing $\phi$ by $\phi (4\pi\tau)^{-\frac{n}{2}}e^{-l_i}$, (\ref{eq_l_6}) becomes
\begin{eqnarray}\label{modifieddistribution}
\int_{\tau_1}^{\tau_2}\int_M\Bigg(\phi\frac{\partial}{\partial\tau}\big((4\pi\tau)^{-\frac{n}{2}}e^{-l_i}\big)+\big\langle\nabla\phi,\nabla\big((4\pi\tau)^{-\frac{n}{2}}e^{-l_i}\big)\big\rangle
\\\nonumber
+\phi R_{g_i}(4\pi\tau)^{-\frac{n}{2}}e^{-l_i}\Bigg)dg_i(\tau)d\tau\leq0,
\end{eqnarray}
where $1<\tau_1<\tau_2<2$ and $\phi$ is a smooth nonnegative function compactly supported on $M\times[\tau_1,\tau_2]$. Combining this and the definition of the reduced volume, we have
\begin{lem}
For any $1<\tau_1<\tau_2<2$ and any smooth nonnegative function $\phi$ compactly supported on $M\times[\tau_1,\tau_2]$ with $\sup\phi\leq 1$, it holds that
\begin{eqnarray}\label{distribution}
\mathcal{V}_i(\tau_2)-\mathcal{V}_i(\tau_1)&\leq&\int_{\tau_1}^{\tau_2}\int_M\Bigg(\phi\frac{\partial}{\partial\tau}\big((4\pi\tau)^{-\frac{n}{2}}e^{-l_i}\big)
\\\nonumber
&&+\big\langle\nabla\phi,\nabla\big((4\pi\tau)^{-\frac{n}{2}}e^{-l_i}\big)\big\rangle+\phi R_{g_i}(4\pi\tau)^{-\frac{n}{2}}e^{-l_i}\Bigg)dg_i(\tau)d\tau\leq0.
\end{eqnarray}
\end{lem}
\begin{proof}
The second inequality of (\ref{distribution}) follows from (\ref{modifieddistribution}).

For the first inequality, let us fix an arbitrary function $\phi$ satisfying the properties indicated in the lemma. Then we may find a positive number $A_0$ large enough, such that
\begin{eqnarray*}
\text{spt}\  \phi \subset B_{g_i(2)}(x_0,A_0)\times[\tau_1,\tau_2].
\end{eqnarray*}
Then, for any $A\geq A_0$, we may apply (\ref{modifieddistribution}) to $\varphi_{A,i}-\phi$, where $\varphi_{A,i}$ is as defined in Lemma \ref{cutoff}. This yields
\begin{eqnarray}\label{01}
\int_{\tau_1}^{\tau_2}\int_M\Bigg(\phi\frac{\partial}{\partial\tau}\big((4\pi\tau)^{-\frac{n}{2}}e^{-l_i}\big)+\big\langle\nabla\phi,\nabla\big((4\pi\tau)^{-\frac{n}{2}}e^{-l_i}\big)\big\rangle
\\\nonumber
+\phi R(4\pi\tau)^{-\frac{n}{2}}e^{-l_i}\Bigg)dg_i(\tau)d\tau
\\\nonumber
\geq\\\nonumber
\int_{\tau_1}^{\tau_2}\int_M\Bigg(\varphi_{A,i}\frac{\partial}{\partial\tau}\big((4\pi\tau)^{-\frac{n}{2}}e^{-l_i}\big)+\big\langle\nabla\varphi_{A,i},\nabla\big((4\pi\tau)^{-\frac{n}{2}}e^{-l_i}\big)\big\rangle
\\\nonumber
+\varphi_{A,i} R_{g_i}(4\pi\tau)^{-\frac{n}{2}}e^{-l_i}\Bigg)dg_i(\tau)d\tau.
\end{eqnarray}
We therefore consider the right-hand-side of the above inequality.

Since $\varphi_{A,i} (4\pi\tau)^{-\frac{n}{2}}e^{-l_i}$ is a Lipschitz function compactly supported on $M\times[\tau_1,\tau_2]$, it is then also absolutely continuous. Hence
\begin{eqnarray}\label{02}
\hspace{-0.2in}\int_M\varphi_{A,i} (4\pi\tau)^{-\frac{n}{2}}e^{-l_i}dg_i(\tau)\ \Bigg\vert_{\tau_1}^{\tau_2}&=&\int_{\tau_1}^{\tau_2}\frac{d}{d\tau}\left(\int_M\varphi_{A,i} (4\pi\tau)^{-\frac{n}{2}}e^{-l_i}dg_i(\tau)\right)
\\\nonumber
&=&\int_{\tau_1}^{\tau_2}\int_M\Bigg(\varphi_{A,i}\frac{\partial}{\partial\tau}\big((4\pi\tau)^{-\frac{n}{2}}e^{-l_i}\big)
\\\nonumber
&&+\varphi_{A,i}R_{g_i}(4\pi\tau)^{-\frac{n}{2}}e^{-l_i}\Bigg)dg_i(\tau)d\tau,
\end{eqnarray}
where we have used the time-independence of $\varphi_{A,i}$. Then the right-hand-sides of (\ref{01}) and (\ref{02}) differ only by the following term.

\begin{eqnarray}\label{03}
&&\left|\int_M\big\langle\nabla\varphi_{A,i},\nabla\big((4\pi\tau)^{-\frac{n}{2}}e^{-l_i}\big)\big\rangle dg_i(\tau)\right|
\\\nonumber
&&\leq\int_M|\nabla\varphi_{A,i}|_{g_i(\tau)}|\nabla l_i|_{g_i(\tau)}(4\pi\tau)^{-\frac{n}{2}}e^{-l_i} dg_i(\tau)
\\\nonumber
&&\leq\sup|\nabla\varphi_{A,i}|_{g_i(\tau)}\left(\int_M|\nabla l_i|^2_{g_i(\tau)}(4\pi\tau)^{-\frac{n}{2}}e^{-l_i} dg(\tau)\right)^{\frac{1}{2}}\left(\int_M(4\pi\tau)^{-\frac{n}{2}}e^{-l_i} dg_i(\tau)\right)^{\frac{1}{2}}
\\\nonumber
&&\leq\frac{C_{K_i}}{A}C_0^{\frac{1}{2}}\big(\mathcal{V}_i(\tau)\big)^{\frac{1}{2}},
\end{eqnarray}
where we have used Lemma \ref{cutoff},  Lemma \ref{grad}, and the Cauchy-Schwarz inequality.

Finally, combining (\ref{01}), (\ref{02}), (\ref{03}), and taking $A\rightarrow\infty$, we prove the first inequality of (\ref{distribution}).
\end{proof}

We are now ready to prove Proposition \ref{propheat}.

\begin{proof}[Proof of Proposition \ref{propheat}]
Let us fix $1<\tau_1<\tau_2<2$ and an arbitrary smooth nonnegative function $\phi$ compactly supported on $M_\infty\times[\tau_1,\tau_2]$ with $\sup\phi\leq 1$. Hence, whenever $i$ is large enough, we have that $\phi\circ\Phi_i^{-1}$ satisfies (\ref{distribution}), where $\Phi_i$ is defined in (\ref{diffeo}). Since $\displaystyle(4\pi\tau)^{-\frac{n}{2}}e^{-l_i}\rightarrow(4\pi\tau)^{-\frac{n}{2}}e^{-l_\infty}$ locally uniformly and in weak $*W^{1,2}\big(\text{spt}\ \phi\big)$ sense, we may take a limit for (\ref{distribution}), and use (\ref{final mono}) to obtain
\begin{eqnarray*}
\int_{\tau_1}^{\tau_2}\int_{M_\infty}\Bigg(\phi\frac{\partial}{\partial\tau}\big((4\pi\tau)^{-\frac{n}{2}}e^{-l_\infty}\big)+\big\langle\nabla\phi,\nabla\big((4\pi\tau)^{-\frac{n}{2}}e^{-l_\infty}\big)\big\rangle
\\
+\phi R_{g_\infty}(4\pi\tau)^{-\frac{n}{2}}e^{-l_\infty}\Bigg)dg_\infty(\tau)d\tau=0.
\end{eqnarray*}
Since $\tau_1$, $\tau_2$, and $\phi$ are arbitrary, this shows that $(4\pi\tau)^{-\frac{n}{2}}e^{-l_\infty}$ is a weak solution to the conjugate heat equation $\displaystyle \frac{\partial }{\partial\tau}u-\Delta u+Ru=0$ on $M_\infty\times(1,2)$. By the standard local regularity theory of linear parabolic equations, we also have that $l_\infty$ is smooth and that $(4\pi\tau)^{-\frac{n}{2}}e^{-l_\infty}$ is a classical solution to the conjugate heat equation.
\end{proof}

Before the proof of the our main theorem, we need one final lemma.

\begin{lem}
$l_\infty$ satisfies the following equations on $M_\infty\times(1,2)$.
\begin{eqnarray}
2\frac{\partial l_\infty}{\partial \tau}+|\nabla l_\infty|^2-R_{g_\infty}+\frac{l_\infty}{\tau}&=&0,\label{lll1}
\\
2\Delta l_\infty-|\nabla l_\infty|^2+R_{g_\infty}+\frac{l_\infty-n}{\tau} &=&0.\label{lll2}
\end{eqnarray}
\end{lem}
\begin{proof}
By Theorem \ref{Perelman}, we have that $l_i$ satisfies (\ref{lll1}) almost everywhere on $M\times[1,2]$. Since $l_i$ is locally Lipschitz, we also have that $l_i$ satisfies (\ref{lll1}) in the sense of distribution.

On the other hand, $l_i\rightarrow l_\infty$ in the $C_{\text{loc}}^{0,\alpha}\big(M_\infty\times(1,2)\big)$ sense as well as in the weak $*W_{\text{loc}}^{1,2}\big(M_\infty\times(1,2)\big)$ sense, therefore (\ref{lll1}) holds for $l_\infty$ in the sense of distribution; see Lemma 9.21 in \cite{MT}, for instance, for the distributional convergence $|\nabla l_i|^2\rightarrow|\nabla l_\infty|^2$, and the distributional convergence of all other terms in (\ref{lll1}) is obvious. By Proposition \ref{propheat}, $l_\infty$ is smooth, hence (\ref{lll1}) holds in the classical sense on $M\times(1,2)$.

Finally, (\ref{lll2}) follows from (\ref{lll1}) and the conjugate heat equation.
\end{proof}

\begin{proof}[Proof of Theorem \ref{asymptotic shrinker}]
Let us define
\begin{eqnarray*}
u&:=&(4\pi\tau)^{-\frac{n}{2}}e^{-l_\infty},
\\
v&:=&\left(\tau(2\Delta l_\infty-|\nabla l_\infty|^2+R_{g_\infty})+l_\infty-n\right)u.
\end{eqnarray*}
Since $u$ satisfies the conjugate heat equation, and according to Perelman's computation (Proposition 9.1 in \cite{P1}), we have
\begin{eqnarray*}
\left(\frac{\partial}{\partial\tau}-\Delta+R_{g_\infty}\right)v=-2\tau\left|Ric_{g_\infty}+\nabla^2l_\infty-\frac{1}{2\tau}g_\infty(\tau)\right|^2u,
\end{eqnarray*}
on $M_\infty\times(1,2)$. By (\ref{lll2}) $v\equiv0$ and $u>0$, it then follows that
\begin{eqnarray*}
Ric_{g_\infty(\tau)}+\nabla^2l_\infty-\frac{1}{2\tau}g_\infty(\tau)=0
\end{eqnarray*}
on $M_\infty\times(1,2)$; this finishes the proof.
\end{proof}

Then, the no shrinking breather theorem follows immediately.

\begin{proof}[Proof of Theorem \ref{main}]
Let us consider the locally uniformly Type I ancient solution as mentioned in Proposition \ref{somethinguseful}. On the one hand, Theorem \ref{asymptotic shrinker} implies
\begin{eqnarray*}
\big(M,g_i(\tau),x_i,l_i\big)_{\tau\in[1,2]}\rightarrow\big(M_\infty,g_\infty(\tau),x_\infty,l_\infty\big)_{\tau\in(1,2)},
\end{eqnarray*}
where the latter is the canonical form of a shrinking gradient Ricci soliton. Note that $l_\infty$ is not defined at $\tau=1$, but $g_\infty$ is. On the other hand, by the final statement of Proposition \ref{somethinguseful}, and since $\tau_i^{-1}\alpha^{-(i+1)}\rightarrow 1$ and $\displaystyle\frac{\tau_{i+1}}{\tau_i}\rightarrow\alpha^{-1}$, we have
\begin{eqnarray*}
(M,g_0(\tau-1),(p_0,1))_{\tau\in[1,\alpha^*)}\cong (M_\infty,g_\infty(\tau),(x_\infty,1))_{\tau\in[1,\alpha^*)},
\end{eqnarray*}
where $\alpha^*=\max\{2,\alpha^{-1}\}>1$; this proves the existence of soliton structure on the original shrinking breather.
\end{proof}

	\section{Further Remarks}

In this section, we continue our consideration of a locally uniformly Type I ancient solution as well as its asymptotic shrinking gradient soliton. All the notations bear the same meaning as in section 6.

\subsection{The Possible Loss of Reduced Volume}

According to Fatou's lemma, we have
\begin{eqnarray}\label{density}
\int_M(4\pi\tau)^{-\frac{n}{2}}e^{-l_\infty}dg_\infty(\tau)&\leq&\lim_{i\rightarrow\infty}\int_{M}(4\pi\tau)^{-\frac{n}{2}}e^{-l_i}dg_i(\tau)
\\\nonumber
&=&\lim_{s\rightarrow\infty}\mathcal{V}(s):=V_\infty.
\end{eqnarray}
Since the left-hand-side is the Gaussian density of a shrinking gradient Ricci soliton, it must always be positive. On the other hand, if $V_\infty=1$, then $(M,g(\tau))$ must be the static Euclidean space; see \cite{YT}. Hence we have
\begin{eqnarray*}
\lim_{s\rightarrow\infty}\mathcal{V}(s):=V_\infty\in(0,1).
\end{eqnarray*}

Now, one may naturally ask: is the inequality in (\ref{density}) always an equality? We do not know the answer yet, though we highly doubt whether the opposite could happen. Recall that in the classical cases: the Type I case and the case of ``curvature positivity'', this equality always holds, and this is because of the following estimate which shows that that the integrand of $\mathcal{V}_i$ is negligible outside compact sets
\begin{eqnarray*}
l_i\sim d_{g_i}^2(x_i,\cdot).
\end{eqnarray*}

\subsection{Noncollapsedness of Locally Uniformly Type I Ancient Solutions}
Next, we observe that a locally uniformly Type I ancient solution is $\kappa$-noncollapsed on all scaled. According to the definition, we only know that such an ancient solution is $\kappa$-noncollpased along a sequence of space-time points $(x_i,\tau_i)$ on scales $\tau_i^{\frac{1}{2}}$. However, the existence of the asymptotic shrinking gradient Ricci soliton enables us to extend the noncollapsedness to everywhere at all scales. We can write formula (\ref{density}) as
\begin{eqnarray*}
\lim_{s\rightarrow\infty}\mathcal{V}_{(p_0,0)}(s)=V_\infty\in(0,1),
\end{eqnarray*}
where we use the subindex to signify the base point.  It then follows from Lemma 3.1 in \cite{YT} that
\begin{eqnarray*}
\lim_{s\rightarrow\infty}\mathcal{V}_{(p,\tau)}(s)\geq V_\infty>0,
\end{eqnarray*}
for all $(p,\tau)\in M\times(0,\infty)$. Note that the proof of  Lemma 3.1 in \cite{YT} requires only time-wise Ricci curvature lower bound. A standard argument as in \cite{P1} implies that $(M,g(\tau))_{\tau\in(0,\infty)}$ is $\kappa$-noncollaposed on all scales, where $\kappa$ depends on the dimension and $V_\infty\in(0,1]$.

\end{document}